\documentclass[11pt]{amsart}
\usepackage{amscd,amssymb}
\usepackage{amsthm,amsmath,amssymb}
\usepackage[matrix,arrow]{xy}

\theoremstyle{definition}
\newtheorem{example}[equation]{Example}

\newtheorem{theorem}[equation]{Theorem}
\newtheorem{lemma}[equation]{Lemma}

\newtheorem*{question*}{Question}

\newtheorem*{problem*}{Problem}

\theoremstyle{remark}
\newtheorem{remark}[equation]{Remark}

\makeatletter\@addtoreset{equation}{section} \makeatother

\author{Ivan Cheltsov}

\title{Del Pezzo surfaces and local inequalities}

\address{\begin{tabbing}
\hspace*{28 em}\=\kill
47/3 Warrender Park Road\\
Edinburgh EH9 1EU, Scotland\\
\\
\texttt{I.Cheltsov@ed.ac.uk}
\end{tabbing}}

\thanks{Throughout this article, I assume that most of considered varieties are
algebraic, normal and defined over complex numbers.}

\begin{document}

\begin{abstract}
I prove new local inequality for divisors on smooth surfaces,
describe its applications, and compare it to a similar local
inequality that is already known by experts.
\end{abstract}

\maketitle

Let $X$ be a Fano variety of dimension $n\geqslant 1$ with at most
Kawamata log terminal singularities (see
\cite[Definition~6.16]{CoKoSm}). In many applications, it is
useful to \emph{measure} how singular effective
$\mathbb{Q}$-divisors $D$ on $X$ can be provided that
$D\sim_{\mathbb{Q}} -K_{X}$. Of course, this can be done in many
ways depending on what I mean by \emph{measure}. A possible
\emph{measurement} can be given by the so-called
$\alpha$-invariant of the Fano variety $X$ that can be defined as
$$
\alpha(X)=\mathrm{sup}\left\{\lambda\in\mathbb{Q}\ \left|%
\aligned
&\text{the pair}\ \left(X, \lambda D\right)\ \text{is Kawamata log terminal}\\
&\text{for every effective $\mathbb{Q}$-divisor}\ D\sim_{\mathbb{Q}} -K_{X}.\\
\endaligned\right.\right\}\in\mathbb{R}.%
$$

The invariant $\alpha(X)$ has been studied intensively by many
people who used different notation for $\alpha(X)$. The notation
$\alpha(X)$ is due to Tian who defined $\alpha(X)$ in a different
way. However, his definition coincides with the one I just gave by
\cite[Theorem~A.3]{ChSh08c}. The $\alpha$-invariants play a very
important role in K\"ahler geometry due to

\begin{theorem}[{\cite{Tian}, \cite[Criterion~6.4]{DeKo01}}] \label{theorem:alpha} Let $X$ be a Fano variety of dimension $n$ that has at most quotient singularities. If
$\alpha(X)>\frac{n}{n+1}$, then $X$ admits an orbifold
K\"ahler--Einstein metric.
\end{theorem}

The $\alpha$-invariants are usually very tricky to compute. But
they are computed in many cases. For example, the
$\alpha$-invariants of smooth del Pezzo surfaces  have been
computed as follows:

\begin{theorem}[{\cite[Theorem~1.7]{Ch07a}}]
\label{theorem:GAFA} Let $S_d$ be a smooth del Pezzo surface of
degree $d$. Then
$$
\aligned
&\alpha(S_d)=\left\{%
\aligned
&\frac{1}{3}\ \mathrm{if}\ d=9, 7 \mbox{ or }  S_d=\mathbb{F}_1,\\%
&\frac{1}{2}\ \mathrm{if}\ d=5, 6 \mbox{ or }  S_d=\mathbb{P}^1\times\mathbb{P}^1,\\%
&\frac{2}{3}\ \mathrm{if}\  d=4,\\%
\endaligned\right.%
\\
&\alpha(S_3)=\left\{%
\aligned
&\frac{2}{3}\ \mathrm{if}\ S_3\ \mathrm{is\ a\ cubic\ surface\ in}\ \mathbb{P}^{3}\ \mathrm{with\ an\ Eckardt\ point},\\%
&\frac{3}{4}\ \mathrm{if}\ S_3\ \mathrm{is\ a\ cubic\ surface\ in}\ \mathbb{P}^{3}\ \mathrm{without\ Eckardt\ points},\\%
\endaligned\right.%
\\
&
\alpha(S_2)=\left\{%
\aligned
&\frac{3}{4}\ \mathrm{if}\  |-K_{S_2}|\ \mathrm{has\ a\ tacnodal\ curve},\\%
&\frac{5}{6}\ \mathrm{if}\   |-K_{S_2}|\ \mathrm{has\ no\ tacnodal\ curves},\\%
\endaligned\right. %
\\
&\alpha(S_1)=\left\{%
\aligned
&\frac{5}{6}\ \mathrm{if}\  |-K_{S_1}|\ \mathrm{has\ a\ cuspidal\ curve},\\%
&1\ \mathrm{if}\  |-K_{S_1}|\ \mathrm{has\ no\ cuspidal\ curves}.\\%
\endaligned\right.\\
\endaligned
$$
\end{theorem}

Note that $\alpha(X)<1$ if and only if there exists an effective
$\mathbb{Q}$-divisor $D$ on $X$ such that $D\sim_{\mathbb{Q}}
-K_{X}$ and the pair $(X,D)$ is not log canonical. Such divisors
(if they exist) are called non-log canonical \emph{special tigers}
by Keel and McKernan (see \cite[Definition~1.13]{KeelMcKernan}).
They play an important role in birational geometry of $X$. How
does one describe non-log canonical special tigers? Note that if
$D$ is a non-log canonical special tiger on $X$, then
$$
(1-\mu)D+\mu D^\prime
$$
is also a  non-log canonical special tiger on $X$ for \emph{any}
effective $\mathbb{Q}$-divisor $D^\prime$ on $X$ such that
$D^\prime\sim_{\mathbb{Q}}-K_{X}$ and \emph{any} sufficiently
small $\mu\geqslant 0$. Thus, to describe  non-log canonical
special tigers on $X$, I only need to consider those of them whose
supports do not contain supports of other non-log canonical
special tigers. Let me call such non-log canonical special tigers
\emph{Siberian tigers}. Unfortunately,  Siberian tigers are not
easy to describe in general. But sometimes it is possible. For
example, Kosta proved

\begin{lemma}[{\cite[Lemma~3.1]{Ko09}}]
\label{lemma:Kosta} Let $S$ be  a hypersurface of degree $6$ in
$\mathbb{P}(1,1,2,3)$ that has exactly one singular point $O$.
Suppose that $O$ is a Du Val singular point of type
$\mathbb{A}_3$. Then all  Siberian tigers on $X$ are cuspidal
curves in $|-K_{S}|$, which implies, in particular, that
$$
\alpha(S)=\left\{%
\aligned
&\frac{5}{6}\ \text{if there is a cuspidal curve in $|-K_{S}|$,}\\
&1 \ \text{otherwise.}\\
\endaligned\right.\\
$$
\end{lemma}

The original proof of Lemma~\ref{lemma:Kosta} is \emph{global} and
lengthy. In \cite{Ko09}, Kosta applied the very same \emph{global}
method to compute the $\alpha$-invariants of all del Pezzo
surfaces of degree $1$ that has at most Du Val singularities (in
most of cases her computations do not give description of
Siberian tigers). Later I noticed that the nature of her
\emph{global} method is, in fact, purely \emph{local}. Implicitly,
Kosta proved

\begin{theorem}[{\cite[Corollary~1.29]{ChK10}}]
\label{theorem:Kosta-original} Let $S$ be a surface, let $P$ be a
smooth point in $S$, let $\Delta_1$ and $\Delta_2$ be two
irreducible curves on $S$ that are both smooth at $P$ and
intersect transversally at $P$, and let $a_1$ and $a_2$ be
non-negative rational numbers. Suppose that
$\frac{2n-2}{n+1}a_{1}+\frac{2}{n+1}a_{2}\leqslant 1$ for some
positive integer $n\geqslant 3$. Let $\Omega$ be an effective
$\mathbb{Q}$-divisor on the~surface $S$ whose support does not
contain the curves $\Delta_1$ and $\Delta_2$. Suppose that the~log
pair $(S,a_{1}\Delta_{1}+a_{2}\Delta_{2}+\Omega)$ is not log
canonical at $P$. Then
$\mathrm{mult}_{P}(\Omega\cdot\Delta_{1})>2a_{1}-a_{2}$ or
$\mathrm{mult}_{P}(\Omega\cdot\Delta_{2})>\frac{n}{n-1}a_{2}-a_{1}$.
\end{theorem}

Unfortunately, Theorem~\ref{theorem:Kosta-original} has a very
limited application scope. Together with Kosta, I generalized
Theorem~\ref{theorem:Kosta-original} as

\begin{theorem}[{\cite[Theorem~1.28]{ChK10}}]
\label{theorem:Kosta} Let $S$ be a surface, let $P$ be a smooth
point in $S$, let $\Delta_1$ and $\Delta_2$ be two irreducible
curves on $S$ that both are smooth at $P$ and intersect
transversally at $P$, let $a_1$ and $a_2$ be non-negative rational
numbers, and let $\Omega$ be an effective $\mathbb{Q}$-divisor on
the~surface $S$ whose support does not contain the curves
$\Delta_1$ and $\Delta_2$. Suppose that the~log pair
$(S,a_{1}\Delta_{1}+a_{2}\Delta_{2}+\Omega)$ is not log canonical
at $P$. Suppose that there are non-negative rational numbers
$\alpha$, $\beta$, $A$, $B$, $M$, and $N$ such that  $\alpha
a_{1}+\beta a_{2}\leqslant 1$, $A(B-1)\geqslant 1$, $M\leqslant
1$, $N\leqslant 1$, $\alpha(A+M-1)\geqslant A^{2}(B+N-1)\beta$,
$\alpha(1-M)+A\beta\geqslant A$. Suppose, in addition, that
$2M+AN\leqslant 2$ or $\alpha(B+1-MB-N)+\beta(A+1-AN-M)\geqslant
AB-1$. Then
$\mathrm{mult}_{P}(\Omega\cdot\Delta_{1})>M+Aa_{1}-a_{2}$ or
$\mathrm{mult}_{P}(\Omega\cdot\Delta_{2})>N+Ba_{2}-a_{1}$.
\end{theorem}

Despite the fact that Theorem~\ref{theorem:Kosta} looks very ugly,
it is much more flexible and much more applicable than
Theorem~\ref{theorem:Kosta-original}. By
\cite[Excercise~6.26]{CoKoSm}, an analogue of
Theorem~\ref{theorem:Kosta} holds for surfaces with at most
quitient singularities. This helped me to find in \cite{Ch13} many
new applications of Theorem~\ref{theorem:Kosta} that do not follow
from Theorem~\ref{theorem:Kosta-original}.

\begin{remark}
\label{remark:applying-theorem-Kosta} How does one apply
Theorem~\ref{theorem:Kosta}? Let me say few words about this. Let
$S$ be a smooth surface, and let $D$ be an effective
$\mathbb{Q}$-divisor on $S$. The purpose of
Theorem~\ref{theorem:Kosta} is to prove that $(S, D)$ is log
canonical provided that $D$ satisfies some \emph{global} numerical
conditions. To do so, I assume that $(S,D)$ is not log canonical
at $P$ and seek for a contradiction. First, I look for some nice
curves that pass through $P$ that has very small intersection with
$D$. Suppose I found two such curves, say $\Delta_1$ and
$\Delta_2$, that are both irreducible and both pass through $P$.
If $\Delta_1$ or $\Delta_2$ are not contained in the support of
the divisor $D$, I can bound $\mathrm{mult}_{P}(D)$  by
$D\cdot\Delta_1$ or $D\cdot\Delta_2$ and, hopefully, get a
contradiction with $\mathrm{mult}_{P}(D)>1$, which follows from
the fact $(S,D)$ is not log canonical at $P$. This shows that I
should look for the curves $\Delta_1$ and $\Delta_2$ among the
curves which are close enough to the boundary of the Mori cone
$\overline{\mathbb{NE}}(S)$. Suppose that both curves $\Delta_1$
and $\Delta_2$ lie in the boundary of the Mori cone
$\overline{\mathbb{NE}}(S)$. Then $\Delta_1^2\leqslant 0$ and
$\Delta_2^2\leqslant 0$. Keeping in mind, that the curves
$\Delta_1$ and $\Delta_2$ can, a priori, be contained in the
support of the divisor $D$, I must put
$D=a_1\Delta_1+a_2\Delta_2+\Omega$ for some non-negative rational
numbers $a_1$ and $a_2$, where $\Omega$ is an effective
$\mathbb{Q}$-divisor on $S$ whose support does not contain the
curves $\Delta_1$ and $\Delta_2$. Then I try to bound $a_1$ and
$a_2$ using some \emph{global} methods. Usually, I end up with two
non-negative rational numbers $\alpha$ and $\beta$ such that
$\alpha a_1+\beta a_2\leqslant 1$. Put $M=D\cdot\Delta_1$,
$N=C\cdot\Delta_2$, $A=-\Delta_1^2$, and $B=-\Delta_1^2$. Suppose
that $\Delta_1$ and $\Delta_2$ are both smooth at $P$ and
intersect transversally at $P$ (otherwise I need to blow up the
surface $S$ and replace the pair $(S,D)$ by its \emph{log pull
back}). If I am lucky, then $A(B-1)\geqslant 1$, $M\leqslant 1$,
$N\leqslant 1$, $\alpha(A+M-1)\geqslant A^{2}(B+N-1)\beta$,
$\alpha(1-M)+A\beta\geqslant A$, and either $2M+AN\leqslant 2$ or
$\alpha(B+1-MB-N)+\beta(A+1-AN-M)\geqslant AB-1$ (or both), which
implies that
$$
M+Aa_{1}-a_{2}\geqslant M+Aa_{1}-a_{2}\Delta_1\cdot\Delta_2=\Omega\cdot\Delta_{1}\geqslant\mathrm{mult}_{P}\Big(\Omega\cdot\Delta_{1}\Big)>M+Aa_{1}-a_{2}%
$$
or
$$
N+Ba_{2}-a_{1}\geqslant
N+Ba_{2}-a_{1}\Delta_1\cdot\Delta_2=\Omega\cdot\Delta_{2}\geqslant\mathrm{mult}_{P}\Big(\Omega\cdot\Delta_{2}\Big)>N+Ba_{2}-a_{1}
$$
by Theorem~\ref{theorem:Kosta}. This is the contradiction I was
looking for.
\end{remark}

Unfortunately, the hypotheses of Theorem~\ref{theorem:Kosta} are
not easy to verify in general. Moreover, the proof of
Theorem~\ref{theorem:Kosta} is very lengthy. It seems likely that
Theorem~\ref{theorem:Kosta} is a special case or, perhaps, a
corollary of a more general statement that looks better and has a
shorter proof. Ideally, the proof of such generalization, if it
exists, should be inductive like the proof of

\begin{theorem}[{\cite[Excercise~6.31]{CoKoSm}}]
\label{theorem:adjunction} Let $S$ be a surface, let $P$ be a
smooth point in $S$, let $\Delta$ be an irreducible curve on $S$
that is smooth at $P$, let $a$ be a non-negative rational number
such that $a\leqslant 1$, and let $\Omega$ be an effective
$\mathbb{Q}$-divisor on the~surface $S$ whose support does not
contain the curve $\Delta$. Suppose that the~log pair
$(S,a\Delta+\Omega)$ is not log canonical at $P$. Then
$\mathrm{mult}_{P}(\Omega\cdot\Delta)>1$.
\end{theorem}

\begin{proof}
Put $m=\mathrm{mult}(\Omega)$. If $m>1$, then I am done, since
$\mathrm{mult}_{P}(\Omega\cdot\Delta)\geqslant m$. In particular,
I may assume that the~log pair $(S,a\Delta+\Omega)$ is log
canonical in a punctured neighborhood of the point $P$. Since
the~log pair $(S,a\Delta+\Omega)$ is not log canonical at $P$,
there exists a birational morphism $h\colon\hat{S}\to S$ that is a
composition of $r\geqslant 1$ blow ups of smooth points dominating
$P$, and there exists an $h$-exceptional divisor, say $E_r$, such
that $e_r>1$, where $e_r$ is a rational number determined by
$$
K_{\hat{S}}+a\hat{\Delta}+\hat{\Omega}+\sum_{i=1}^{r}e_iE_i\sim_{\mathbb{Q}}h^{*}\big(K_{S}+a\Delta+\Omega\big),%
$$
where each $e_i$ is a rational number, each $E_i$ is an
$h$-exceptional divisor,  $\hat{\Omega}$ is a proper transform on
$\hat{S}$ of the divisor $\Omega$, and $\hat{\Delta}$ is a proper
transform on $\hat{S}$ of the curve $\Delta$.

Let $f\colon \tilde{S}\to S$ be the blow up of the point $P$, let
$\tilde{\Omega}$ be the proper transform of the divisor $\Omega$
on the surface $\tilde{S}$, let $E$ be the $f$-exceptional curve,
and let $\tilde{\Delta}$ be the proper transform of the curve
$\Delta$ on the surface $\tilde{S}$. Then the log pair
$(\tilde{S}, a\tilde{\Delta}+(a+m-1)E+\tilde{\Omega})$ is not log
canonical at some point $Q\in E$.

Let me prove the inequality
$\mathrm{mult}_{P}(\Omega\cdot\Delta)>1$ by induction on $r$. If
$r=1$, then $a+m-1>1$, which implies that $m>2-a\geqslant 1$. This
implies that $\mathrm{mult}_{P}(\Omega\cdot\Delta)>1$ if $r=1$.
Thus, I may assume that $r\geqslant 2$. Since
$$
\mathrm{mult}_{P}\Big(\Omega\cdot\Delta\Big)\geqslant m+\mathrm{mult}_{Q}\Big(\tilde{\Omega}\cdot\tilde{\Delta}\Big),%
$$
it is enough to prove that
$m+\mathrm{mult}_{Q}(\tilde{\Omega}\cdot\tilde{\Delta})>1$.
Moreover, I may assume that $m\leqslant 1$, since
$\mathrm{mult}_{P}(\Omega\cdot\Delta)\geqslant m$. Then the log
pair $(\tilde{S}, a\tilde{\Delta}+(a+m-1)E+\tilde{\Omega})$ is log
canonical at a punctured neighborhood of the point $Q\in E$, since
$a+m-1\leqslant 2$.

If $Q\not\in\tilde{\Delta}$, then the~log pair
$(\tilde{S},(a+m-1)E+\tilde{\Omega})$ is not log canonical at
the~point $Q$, which implies that
$$
m=\tilde{\Omega}\cdot
E\geqslant\mathrm{mult}_{Q}\Big(\tilde{\Omega}\cdot E\Big)>1
$$
by
induction. The latter implies that $Q=\tilde{\Delta}\cap E$, since
$m\leqslant 1$. Then
$$
a+m-1+\mathrm{mult}_{Q}\Big(\tilde{\Omega}\cdot\tilde{\Delta}\Big)=\mathrm{mult}_{Q}\Bigg(\Big((a+m-1)E+\tilde{\Omega}\Big)\cdot\tilde{\Delta}\Bigg)>1%
$$
by induction. This implies that
$\mathrm{mult}_{Q}(\tilde{\Omega}\cdot\tilde{\Delta})>2-a-m$. Then
$m+\mathrm{mult}_{Q}(\tilde{\Omega}\cdot\tilde{\Delta})>2-a\geqslant
1$ as required.
\end{proof}

Recently, I jointly with Park and Won proved that all Siberian
tigers on smooth cubic surfaces are just anticanonical curves that
have non-log canonical singularities (see
\cite[Theorem~1.12]{CheltsovParkWon2}). This follows from

\begin{theorem}[{\cite[Corollary~1.13]{CheltsovParkWon2}}]
\label{theorem:cubic-surface} Let $S$ be a smooth cubic surface in
$\mathbb{P}^3$, let $P$ be a point in $S$, let $T_P$ be the unique
hyperplane section of the surface $S$ that is singular at $P$, let
$D$ be any effective $\mathbb{Q}$-divisor on the surface $S$ such
that $D\sim_{\mathbb{Q}} -K_{S}$. Then $(S,D)$ is log canonical at
$P$ provided that $\mathrm{Supp}(D)$ does not contain at least one
irreducible component of $\mathrm{Supp}(T_P)$.
\end{theorem}

Siberian tigers on smooth del Pezzo surfaces of degree $1$ and $2$
are also just anticanonical curves that have non-log canonical
singularities (see \cite[Theorem 1.12]{CheltsovParkWon2}). This
follows easily from the proofs of
\cite[Lemmas~3.1~and~3.5]{Ch07a}. Surprisingly, smooth del Pezzo
surfaces of degree $4$ contains much more  Siberian tigers.

\begin{example}
\label{example:dP4} Let $S$ be a smooth complete intersection of
two quadric hypersurfaces in $\mathbb{P}^4$, let $L$ be a line on
$S$, and let $P_0$ be a point in $L$ such that $L$ is the only
line in $S$ that passes though $P_0$. Then there exists exactly
five conics in $S$ that pass through $P_0$. Let me denote them by
$C_1^0$, $C_2^0$, $C_3^0$, $C_4^0$ and $C_5^0$. Then
$$
\frac{\sum_{i=1}^{5}C_i^0}{3}+\frac{2}{3}L\sim_{\mathbb{Q}}-K_{S},
$$
is a  Siberian tiger. Let $Z$ be a general smooth rational cubic
curve in $S$ such that $Z+L$ is cut out by a hyperplane section
and $P\in Z$. Then $Z\cap L$ consists of a point $P$ and another
point which I denote by $Q$. Let $f\colon \tilde{S}\to S$ be a
blow up of the point $Q$, and let $E$ be its exceptional curve.
Denote by $\tilde{L}$ and $\tilde{Z}$ the proper transforms of the
curves $L$ and $Z$ on the surface $\tilde{S}$, respectively. Then
$\tilde{Z}\cap\tilde{L}=\varnothing$. Let $g\colon\hat{S}\to
\tilde{S}$ be the blow up of the point $\tilde{Z}\cap E$, and let
$F$ be its exceptional curve. Denote by $\hat{E}$, $\hat{L}$ and
$\hat{Z}$ the proper transforms of the curves $E$, $\tilde{L}$ and
$\tilde{Z}$ on the surface $\hat{S}$, respectively. Then $\hat{S}$
is a minimal resolution of a singular del Pezzo surface of degree
$2$, and $|-K_{\hat{S}}|$ gives a morphism
$\hat{S}\to\mathbb{P}^2$ that is a double cover away from the
curves $\hat{E}$ and $\hat{L}$. This double cover induces an
involution $\tau\in\mathrm{Bir}(S)$. Put $C_i^1=\tau(C_i^0)$ for
every $i$. Then $C_1^1$, $C_2^1$, $C_3^1$, $C_4^1$ and $C_5^1$ are
curves of degree $5$ that all intersect exactly in one point in
$L$. Denote this point by $P_1$. Iterate this constriction $k$
times. This gives me five irreducible curves $C_1^k$, $C_2^k$,
$C_3^k$, $C_4^k$ and $C_5^k$ that intersect exactly in one point
$P_k$. Then
\begin{equation}
\label{equation:tiger}
\frac{\sum_{i=1}^{5}C_i^k}{a_{2k+1}+a_{2k+3}}+\frac{4a_{2k+1}-a_{2k+3}}{a_{2k+1}+a_{2k+3}}L\sim_{\mathbb{Q}}-K_{S},
\end{equation}
where $a_i$ is the $i$-th Fibonacci number. Moreover, each curve
$C_i^k$ is a curve of degree $a_{2k+3}$. Furthermore, the log
canonical threshold of the divisor ~\eqref{equation:tiger} is
$$
\frac{a_{2k+3}(a_{2k+1}+a_{2k+3})}{1+a_{2k+3}(a_{2k+1}+a_{2k+3})}<1,
$$
which easily implies that the divisor~\eqref{equation:tiger} is a
 Siberian tiger.
\end{example}

Quite surprisingly, Theorem~\ref{theorem:cubic-surface} has other
applications as well. For example, it follows from
\cite[Corollary~2.12]{KPZ12a}, \cite[Lemma~1.10]{CheltsovParkWon2}
and Theorem~\ref{theorem:cubic-surface} that every cubic cone in
$\mathbb{A}^4$ having unique singular point dooes not admit
non-trivial regular $\mathbb{G}_{a}$-actions (cf.
\cite[Question~2.22]{FZ2003}).

The crucial part in the proof of
Theorem~\ref{theorem:cubic-surface} is played by two sibling
lemmas. The first one is

\begin{lemma}[{\cite[Lemma~4.8]{CheltsovParkWon2}}]
\label{lemma:cubic-surface} Let $S$ be a smooth cubic surface in
$\mathbb{P}^3$, let $P$ be a point in $S$, let $T_P$ be the unique
hyperplane section of the surface $S$ that is singular at $P$, let
$D$ be any effective $\mathbb{Q}$-divisor on the surface $S$ such
that $D\sim_{\mathbb{Q}} -K_{S}$. Suppose that $T_{P}$ consists of
three lines such that one of them does not pass through $P$. Then
$(S,D)$ is log canonical at $P$.
\end{lemma}

Its younger sister is

\begin{lemma}[{\cite[Lemma~4.9]{CheltsovParkWon2}}]
\label{lemma:cubic-surface-2}  Let $S$ be a smooth cubic surface
in $\mathbb{P}^3$, let $P$ be a point in $S$, let $T_P$ be the
unique hyperplane section of the surface $S$ that is singular at
$P$, let $D$ be any effective $\mathbb{Q}$-divisor on the surface
$S$ such that $D\sim_{\mathbb{Q}} -K_{S}$.  Suppose that $T_{P}$
consists of a line and a conic intersecting transversally. Then
$(S,D)$ is log canonical at $P$.
\end{lemma}

The proofs of Lemmas~\ref{lemma:cubic-surface} and
\ref{lemma:cubic-surface-2} we found in \cite{CheltsovParkWon2}
are \emph{global}. In fact, they resemble the proofs of classical
results by Segre and Manin on cubic surfaces (see
\cite[Theorems~2.1~and~2.2]{CoKoSm}). Once the paper
\cite{CheltsovParkWon2} has been written, I asked myself a
question: can I prove Lemmas~\ref{lemma:cubic-surface} and
\ref{lemma:cubic-surface-2} using just \emph{local} technique? To
answer this question, let me sketch their \emph{global} proofs
first.

\begin{proof}[Global proof of Lemma~\ref{lemma:cubic-surface}]
Let me use the notation and assumptions of
Lemma~\ref{lemma:cubic-surface}. I write $T_P=L+M+N$, where $L$,
$M$, and $N$ are lines on the cubic surface $S$. Without loss of
generality, I may assume that the line $N$ does not pass through
the point $P$. Let $D$ be any effective $\mathbb{Q}$-divisor on
the surface $S$ such that $D\sim_{\mathbb{Q}} -K_{S}$. I must show
that  $(S,D)$ is log canonical at $P$. Suppose that the log pair
$(S,D)$ is not log canonical at  the point $P$. Let me seek for a
contradiction.

Put $D=aL+bM+cN+\Omega$, where $a$, $b$, and $c$ are non-negative
rational numbers and $\Omega$ is an effective $\mathbb{Q}$-divisor
on $S$ whose support contains none of the lines $L$, $M$ and $N$.
Put $m=\mathrm{mult}_P(\Omega)$. Then $a\leqslant 1$, $b\leqslant
1$ and $c\leqslant 1$. Moreover, the pair $(S,D)$ is log canonical
outside finitely many points. This follows from
\cite[Lemma~5.3.6]{CoKoSm} and is very easy to prove (see, for
example, \cite[Lemma~4.1]{CheltsovParkWon2} or the proof of
\cite[Lemma~3.4]{Ch07a}).

Since $(S,D)$ is not log canonical at the point $P$, I have
$$
m+a+b=\mathrm{mult}_{P}(D)>1
$$
by \cite[Excercise~6.18]{CoKoSm} (this also follows from
Theorem~\ref{theorem:adjunction}). In particular, the rational
number $a$ must be positive, since otherwise I would have
$$
1=L\cdot D\geqslant\mathrm{mult}_{P}(D)>1.
$$
Similarly, the
rational number $b$ must be positive as well.

The inequality $m+a+b>1$ is very handy. However, a stronger
inequality $m+a+b>c+1$ holds. Indeed, there exists a non-negative
rational number $\mu$ such that the divisor $(1+\mu)D-\mu T_P$ is
effective and its support does not contain at least one components
of $T_P$. Now to obtain $m+a+b>c+1$, it is enough to apply
\cite[Excercise~6.18]{CoKoSm} to the divisor $(1+\mu)D-\mu T_P$,
since $(S, (1+\mu)D-\mu T_P)$ is not log canonical at $P$.

Since $a$, $b$, $c$ do not exceed $1$ and $(S, L+M+N)$ is log
canonical, $\Omega\ne 0$. Let me write
$\Omega=\sum_{i=1}^{r}e_{i}C_{i}$, where every $e_{i}$ is a
positive rational number, and every $C_{i}$ is an irreducible
reduced curve of degree $d_{i}>0$ on the surface $S$. Then
$$
a+b+c+\sum_{i=1}^{r}e_{i}d_{i}=3,
$$
since $-K_{S}\cdot D=3$.

Let $f\colon\tilde{S}\to S$ be a blow up of the point $P$, and let
$E$ be the exceptional divisor of $f$. Denote by $\tilde{L}$,
$\tilde{M}$ and $\tilde{N}$ the proper transforms on $\tilde{S}$
of the lines $L$, $M$ and $N$, respectively. For each $i$, denote
by $\tilde{C}_{i}$ the proper transform of the curve $C_{i}$ on
the surface $\tilde{S}$. Then
$$
K_{\tilde{S}}+a\tilde{L}+b\tilde{M}+c\tilde{N}+\left(a+b+m-1\right)E+\sum_{i=1}^{r}e_{i}\tilde{C}_{i}=f^{*}\left(K_{S}+D\right),
$$
which implies that the log pair $(\tilde{S},
a\tilde{L}+b\tilde{M}+c\tilde{N}+(a+b+m-1)E+\sum_{i=1}^{r}e_{i}\tilde{C}_{i})$
is not log canonical at some point $Q\in E$.

I claim that either $Q\in\tilde{L}\cap E$ or $Q\in\tilde{M}\cap
E$. Indeed, it follows from
$$
\left\{\aligned%
&1=D\cdot L=\Big(aL+bM+cN+\Omega\Big)\cdot L=-a+b+c+\Omega\cdot L\geqslant -a+b+c+m,\\
&1=D\cdot M=\Big(aL+bM+cN+\Omega\Big)\cdot M=a-b+c+\Omega\cdot M\geqslant a-b+c+m,\\
&1=D\cdot N=\Big(aL+bM+cN+\Omega\Big)\cdot N=a+b-c+\Omega\cdot N\geqslant a+b-c,\\
\endaligned
\right.
$$
that $m\leqslant 1-c$ and $a+b+m-1\leqslant 1$, because
$a\leqslant 1$ and $b\leqslant 1$. On the other hand, if
$Q\not\in\tilde{L}\cup\tilde{M}$, then the log pair $(\tilde{S},
\left(a+b+m-1\right)E+\sum_{i=1}^{r}e_{i}\tilde{C}_{i})$ is not
log canonical at $Q$, which implies that
$$
m=\Big(\sum_{i=1}^{r}e_{i}\tilde{C}_{i}\Big)\cdot E>1
$$
by Theorem~\ref{theorem:adjunction}. This shows that either
$Q\in\tilde{L}\cap E$ or $Q\in\tilde{M}\cap E$, since $m\leqslant
1-c\leqslant 1$. Without loss of generality, I may assume that
$Q=\tilde{L}\cap E$.

Let $\rho\colon S\dasharrow\mathbb{P}^2$ be the linear projection
from the point $P$. Then $\rho$ is a generically two-to-one
rational map. Thus the map $\rho$ induces an involution
$\tau\in\mathrm{Bir}(S)$ known as the Geiser involution (see
\cite[\S~2.14]{CoKoSm}). The involution $\tau$ is biregular
outside $P\cup N$, $\tau(L)=L$ and $\tau(M)=M$.

 For each $i$, denote by $\hat{d}_{i}$
the degree of the curve $\tau(C_{i})$. Put
$\hat{\Omega}=\sum_{i=1}^re_i\tau(C_{i})$. Then
$$
aL+bM+(a+b+m-1)N+\hat{\Omega}\sim_{\mathbb{Q}}-K_{S},
$$
and $(S,aL+bM+(a+b+m-1)M+\hat{\Omega})$ is not log canonical at
the point $L\cap N$. Thus, I can replace the original effective
$\mathbb{Q}$-divisor $D$ by the divisor
$$
aL+bM+(a+b+m-1)N+\hat{\Omega}\sim_{\mathbb{Q}} -K_{S}
$$
that has the same properties as $D$. Moreover, I have
$$
\sum_{i=1}^{r}e_i\hat{d}_{i}<\sum_{i=1}^{r}e_id_{i},
$$
since $m+a+b>c+1$. Iterating this process, I obtain a
contradiction after finitely many steps.
\end{proof}

\begin{proof}[Global proof of Lemma~\ref{lemma:cubic-surface-2}]
Let me use the notations and assumptions of
Lemma~\ref{lemma:cubic-surface-2}. I write $T_P=L+C$, where $L$ is
a line, and $C$ is a conic. Let $D$ be any effective
$\mathbb{Q}$-divisor on the surface $S$ such that
$D\sim_{\mathbb{Q}} -K_{S}$. I must show that the log pair $(S,D)$
is log canonical at $P$. Suppose that $(S,D)$ is not log canonical
at the point $P$. Let me seek for a contradiction.

Let me write $D=nL+kC+\Omega$, where $n$ and $k$ are non-negative
rational numbers and $\Omega$ is an effective $\mathbb{Q}$-divisor
on $S$ whose support contains none of the curves $L$ and $C$. Put
$m=\mathrm{mult}_P(\Omega)$. Then $2n+m\leqslant 2$ and
$2k+m\leqslant 1+n$, since
$$
\left\{\aligned%
&1=D\cdot L=\Big(nL+kC+\Omega\Big)\cdot L=-n+2k+\Omega\cdot L\geqslant -n+2k+m,\\
&2=D\cdot C=\Big(nL+kC+\Omega\Big)\cdot C=2n+\Omega\cdot C\geqslant 2n+m.\\
\endaligned
\right.
$$

Arguing as in the proof \cite[Lemma~3.4]{Ch07a}, I see that the
log pair $(S,D)$ is log canonical outside finitely many points
(this follows, for example, from \cite[Lemma~5.3.6]{CoKoSm}). In
particular, both rational numbers $n$ and $k$ do not exceed $1$.
On the other hand, it follows from \cite[Excercise~6.18]{CoKoSm}
that
$$
m+n+k=\mathrm{mult}_{P}(D)>1,
$$
because the log pair $(S,D)$ is not log canonical at the point
$P$. The later implies that $n>0$, since $1=L\cdot
D\geqslant\mathrm{mult}_{P}(D)$ if $n=0$.

I claim that $n>k$ and $m+n>1$. Indeed, there exists a
non-negative rational number $\mu$ such that the divisor
$(1+\mu)D-\mu T_P$ is effective and its support does not contain
at least one components of $T_P$. Then $(S, (1+\mu)D-\mu T_P)$ is
not log canonical at $P$. If $n\leqslant k$, then the support of
$(1+\mu)D-\mu T_P$ does not contain $L$, which is impossible,
since
$$
\mathrm{mult}_P\Big((1+\mu)D-\mu T_P\Big)>1
$$
and $1=L\cdot((1+\mu)D-\mu T_P)$. Thus, I proved that $n>k$. Now I
can apply \cite[Excercise~6.18]{CoKoSm} to the divisor
$(1+\mu)D-\mu T_P$ and obtain $m+n>1$.

Let $f\colon \tilde{S}\to S$ be the blow up of the point $P$, let
$\tilde{\Omega}$ be the proper transform of the divisor $\Omega$
on the surface $\tilde{S}$, let $\tilde{L}$ be the proper
transform of the line $L$ on the surface $\tilde{S}$, let
$\tilde{C}$ be the proper transform of the conic $C$ on the
surface $\tilde{S}$, and let $E$ be the $f$-exceptional curve.
Then
$$
K_{\tilde{S}}+n\tilde{L}+k\tilde{C}+\tilde{\Omega}+\big(n+k+m-1\big)E\sim_{\mathbb{Q}}f^{*}\big(K_{S}+D\big)\sim_{\mathbb{Q}} 0,%
$$
which implies that the log pair $(\tilde{S},
n\tilde{L}+k\tilde{C}+(n+k+m-1)E+\tilde{\Omega})$ is not log
canonical at some point $Q\in E$. One the other hand, I must have
$n+k+m-1\leqslant 1$, because $2n+m\leqslant 2$, $2k+m\leqslant
1+n$ and $n\leqslant 1$.

I claim that $Q\in\tilde{L}$. Indeed, if $Q\in\tilde{C}$, then the
log pair $(\tilde{S}, k\tilde{C}+(n+k+m-1)E+\tilde{\Omega})$ is
not log canonical at $Q$, which implies that $k>n$, since
$$
1-n+k=\Big(\tilde{\Omega}+\big(n+k+m-1\big)E\Big)\cdot\tilde{C}>1,
$$
by Theorem~\ref{theorem:adjunction}. Since I proved already that
$n>k$, the curve $\tilde{C}$ does not contain $Q$. Thus, if
$Q\not\in\tilde{L}$, then $Q\not\in\tilde{L}\cup\tilde{C}$, which
contradicts \cite[Lemma~3.2]{CheltsovParkWon2}, since
$$
n\tilde{L}+k\tilde{C}+\tilde{\Omega}+(n+k+m-1)E\sim_{\mathbb{Q}}-K_{\tilde{S}}.
$$

Since $n$ and $k$ do not exceed $1$ and the log pair $(S, L+C)$ is
log canonical, the effective $\mathbb{Q}$-divisor $\Omega$ cannot
be the zero-divisor. Let $r$ be the number of the irreducible
components of the support of the $\mathbb{Q}$-divisor $\Omega$.
Let me write $\Omega=\sum_{i=1}^{r}e_{i}C_{i}$, where every
$e_{i}$ is a positive rational number, and every $C_{i}$ is an
irreducible reduced curve of degree $d_{i}>0$ on the surface $S$.
Then
$$
n+2k+\sum_{i=1}^{r}a_{i}d_i=3,
$$
since $-K_{S}\cdot D=3$.

Let $\rho\colon S\dasharrow\mathbb{P}^2$ be the linear projection
from the point $P$. Then $\rho$ is a generically $2$-to-$1$
rational map. Thus the map $\rho$ induces a birational involution
$\tau$ of the cubic surface $S$. This involution is also known as
the Geiser involution (cf. the proof of
Lemma~\ref{lemma:cubic-surface}). The involution $\tau$ is
biregular outside of the conic $C$, and $\tau(L)=L$.

For every $i$, put $\hat{C}_i=\tau(C_i)$, and denote by
$\hat{d}_i$ the degree of the curve $\hat{C}_i$. Put
$\hat{\Omega}=\sum_{i=1}^re_i\hat{C}_{i}$. Then
$$
nL+(n+k+m-1)C+\hat{\Omega}\sim_{\mathbb{Q}}-K_{S},
$$
and $(S,nL+(n+k+m-1)C+\hat{\Omega})$ is not log canonical at the
point $L\cap C$ that is different from $P$. Thus, I can replace
the original effective $\mathbb{Q}$-divisor $D$ by
$nL+(n+k+m-1)C+\hat{\Omega}$ that has the same properties as $D$.
Moreover, since $m+n>1$, the inequality
$$
\sum_{i=1}^{r}e_i\hat{d}_{i}<\sum_{i=1}^{r}e_id_{i}
$$
holds.
Iterating this process, I obtain a contradiction in a finite
number of steps as in the proof of
Lemma~\ref{lemma:cubic-surface}.
\end{proof}

It came as a surprise that Theorem~\ref{theorem:Kosta} can be used
to replace the \emph{global} proof of
Lemma~\ref{lemma:cubic-surface-2} by its \emph{local} counterpart.
Let me show how to do this.

\begin{proof}[{Local proof of Lemma~\ref{lemma:cubic-surface-2}}]
Let me use the assumptions and notation of
Lemma~\ref{lemma:cubic-surface-2}. I write $T_P=L+C$, where $L$ is
a line, and $C$ is a conic. Let $D$ be any effective
$\mathbb{Q}$-divisor on the surface $S$ such that
$D\sim_{\mathbb{Q}} -K_{S}$. I must show that the log pair $(S,D)$
is log canonical at $P$. Suppose that $(S,D)$ is not log canonical
at the point $P$. Let me seek for a contradiction.

Put $D=nL+kC+\Omega$, where $n$ and $k$ are non-negative rational
numbers and $\Omega$ is an effective $\mathbb{Q}$-divisor on $S$
whose support contains none of the curves $L$ and $C$. Put
$m=\mathrm{mult}_P(\Omega)$. Then
$$
m+n+k=\mathrm{mult}_{P}(D)>1,
$$
since $(S,D)$ is not log canonical
at  $P$. The later implies that $n>0$, since $1=L\cdot
D\geqslant\mathrm{mult}_{P}(D)$ if $n=0$.

Replacing $D$ by an effective $\mathbb{Q}$-divisor $(1+\mu)D-\mu
T_P$ for an appropriate $\mu\geqslant 0$, I may assume that $k=0$.
Then $2=C\cdot D=2n+\Omega\cdot C\geqslant 2n+m$. Moreover, the
log pair $(S,D)$ is log canonical outside finitely many points.
The latter follows, for example, from \cite[Lemma~5.3.6]{CoKoSm}
and is very easy to prove (cf. the proof of
\cite[Lemma~3.4]{Ch07a}).

Let $f\colon \tilde{S}\to S$ be the blow up of the point $P$, let
$\tilde{\Omega}$ be the proper transform of the divisor $\Omega$
on the surface $\tilde{S}$, let $\tilde{L}$ be the proper
transform of the line $L$ on the surface $\tilde{S}$, and let $E$
be the $f$-exceptional curve. Then
$$
K_{\tilde{S}}+n\tilde{L}+\tilde{\Omega}+\big(n+m-1\big)E\sim_{\mathbb{Q}}f^{*}\big(K_{S}+D\big)\sim_{\mathbb{Q}} 0,%
$$
which implies that $(\tilde{S},
n\tilde{L}+(n+m-1)E+\tilde{\Omega})$ is not log canonical at some
point $Q\in E$. Arguing as in the proof of
\cite[Lemma~3.5]{Ch07a}, I get $Q=\tilde{L}\cap E$. Now I can
apply Theorem~\ref{theorem:Kosta} to the log pair $(\tilde{S},
n\tilde{L}+(n+m-1)E+\tilde{\Omega})$ at the point $Q$. Put
$\Delta_1=E$, $\Delta_2=\tilde{L}$, $M=1$, $A=1$, $N=0$, $B=2$,
and $\alpha=\beta=1$. Check that all hypotheses of
Theorem~\ref{theorem:Kosta} are satisfied. By
Theorem~\ref{theorem:Kosta}, I have
$$
m=\mathrm{mult}_Q(\tilde{\Omega}\cdot E)>1+(n+m-1)-n=m
$$
or
$1+n-m=\mathrm{mult}_Q(\tilde{\Omega}\cdot\tilde{L})>2n-(n+m-1)=1+n-m
$, which is absurd.
\end{proof}

I tried to apply Theorem~\ref{theorem:Kosta} to find a
\emph{local} proof of Lemma~\ref{lemma:cubic-surface} as well. But
I failed. This is not surprising. Let me explain why. The proof of
Theorem~\ref{theorem:Kosta} is \emph{asymmetric} with respect to
the curves $\Delta_1$ and $\Delta_2$. The \emph{global} proof of
Lemma~\ref{lemma:cubic-surface-2} is also \emph{asymmetric} with
respect to the curves $L$ and $C$. The proof of
Theorem~\ref{theorem:Kosta} is based on uniquely determined
iterations of blow ups: I must keep blowing up the the point of
the proper transform of the curve $\Delta_2$ that dominates the
point $P$. The  \emph{global} proof of
Lemma~\ref{lemma:cubic-surface-2} is based on uniquely determined
composition of Geiser involutions. So,
Lemma~\ref{lemma:cubic-surface-2} can be considered as a
\emph{global} wrap up of a purely \emph{local} special case of
Theorem~\ref{theorem:Kosta}, where the line $L$ plays the role of
the curve $\Delta_2$ in Theorem~\ref{theorem:Kosta}. On the other
hand, Lemma~\ref{lemma:cubic-surface} is \emph{symmetric} with
respect to the lines $L$ and $M$. Moreover, its proof is not
deterministic at all, since the composition of Geiser involutions
in the proof of Lemma~\ref{lemma:cubic-surface} is not uniquely
determined by the initial data, i.e. every time I apply Geiser
involution, I have exactly two possible candidates for the next
one: either I can use the Geiser involution induced by the
projection from $L\cap N$ or I can use the Geiser involution
induced by the projection from $M\cap N$. So, there is a little
hope that Theorem~\ref{theorem:Kosta} can be used to replace the
usage of Geiser involutions in the proof of
Lemma~\ref{lemma:cubic-surface}. Of course, there is a chance that
the proof of Lemma~\ref{lemma:cubic-surface} can not be
\emph{localized} like the proof of
Lemma~\ref{lemma:cubic-surface-2}. Fortunately, this is not the
case. Indeed, instead of using Geiser involutions in the
\emph{global} proof of Lemma~\ref{lemma:cubic-surface}, I can use

\begin{theorem}
\label{theorem:main} Let $S$ be a surface, let $P$ be a smooth
point in $S$, let $\Delta_1$ and $\Delta_2$ be two irreducible
curves on $S$ that both are smooth at $P$ and intersect
transversally at $P$, let $a_1$ and $a_2$ be non-negative rational
numbers, and let $\Omega$ be an effective $\mathbb{Q}$-divisor on
the~surface $S$ whose support does not contain the curves
$\Delta_1$ and $\Delta_2$. Suppose that the~log pair
$(S,a_{1}\Delta_{1}+a_{2}\Delta_{2}+\Omega)$ is not log canonical
at $P$. Put $m=\mathrm{mult}_{P}(\Omega)$. Suppose that
$m\leqslant 1$. Then
$\mathrm{mult}_{P}(\Omega\cdot\Delta_{1})>2(1-a_{2})$ or
$\mathrm{mult}_{P}(\Omega\cdot\Delta_{2})>2(1-a_{1})$.
\end{theorem}

\begin{proof}
I may assume that $a_1\leqslant 1$ and $a_2\leqslant 1$. Then
the~log pair $(S,a_{1}\Delta_{1}+a_{2}\Delta_{2}+\Omega)$ is log
canonical in a punctured neighborhood of the point $P$, because
$m\leqslant 1$.

Since the~log pair $(S,a_{1}\Delta_{1}+a_{2}\Delta_{2}+\Omega)$ is
not log canonical at $P$, there exists a birational morphism
$h\colon\hat{S}\to S$ that is a composition of $r\geqslant 1$ blow
ups of smooth points dominating $P$, and there exists an
$h$-exceptional divisor, say $E_r$, such that $e_r>1$, where $e_r$
is a rational number determined by
$$
K_{\hat{S}}+a_1\hat{\Delta}_1+a_2\hat{\Delta}_2+\hat{\Omega}+\sum_{i=1}^{r}e_iE_i\sim_{\mathbb{Q}}h^{*}\big(K_{S}+a_{1}\Delta_{1}+a_{2}\Delta_{2}+\Omega\big),%
$$
where $e_i$ is a rational number, each $E_i$ is an $h$-exceptional
divisor,  $\hat{\Omega}$ is a proper transform on $\hat{S}$ of the
divisor $\Omega$, $\hat{\Delta}_1$ and $\hat{\Delta}_2$, are
proper transforms on $\hat{S}$ of the curves $\Delta_1$ and
$\Delta_2$, respectively.

Let $f\colon \tilde{S}\to S$ be the blow up of the point $P$, let
$\tilde{\Omega}$ be the proper transform of the divisor $\Omega$
on the surface $\tilde{S}$, let $E$ be the $f$-exceptional curve,
let $\tilde{\Delta}_1$ and $\tilde{\Delta}_2$ be the proper
transforms of the curves $\Delta_1$ and $\Delta_2$ on the surface
$\tilde{S}$, respectively. Then
$$
K_{\tilde{S}}+a_1\tilde{\Delta}_1+a_2\tilde{\Delta}_2+\big(a_1+a_2+m-1\big)E+\tilde{\Omega}\sim_{\mathbb{Q}}f^{*}\big(K_{S}+a_{1}\Delta_{1}+a_{2}\Delta_{2}+\Omega\big).%
$$
which implies that the log pair $(\tilde{S},
a_1\tilde{\Delta}_1+a_2\tilde{\Delta}_2+\big(a_1+a_2+m-1\big)E+\tilde{\Omega})$
is not log canonical at some point $Q\in E$.

If $r=1$, then $a_1+a_2+m-1>1$, which implies that $m>2-a_1-a_2$.
On the other hand, if  $m>2-a_1-a_2$, then either $m>2(1-a_1)$ or
$m>2(1-a_2)$, because otherwise I would have $2m\leqslant
4-2(a_1+a_2)$, which contradicts to $m>2-a_1-a_2$. Thus, if $r=1$,
them $\mathrm{mult}_{P}(\Omega\cdot\Delta_{1})>2(1-a_{2})$ or
$\mathrm{mult}_{P}(\Omega\cdot\Delta_{2})>2(1-a_{1})$.

Let me prove the required assertion by induction on $r$. The case
$r=1$ is done. Thus, I may assume that $r\geqslant 2$. If $Q\ne
E\cap\tilde{\Delta}_1$ and $Q\ne E\cap\tilde{\Delta}_2$, then it
follows from Theorem~\ref{theorem:adjunction} that
$m=\tilde{\Omega}\cdot E>1$, which is impossible, since
$m\leqslant 1$ by assumption. Thus, either
$Q=E\cap\tilde{\Delta}_1$ or $Q=E\cap\tilde{\Delta}_2$. Without
loss of generality, I may assume that $Q=E\cap\tilde{\Delta}_1$.

By induction, I can apply the required assertion to $(\tilde{S},
a_1\tilde{\Delta}_1+(a_1+a_2+m-1)E+\tilde{\Omega})$ at the point
$Q$. This implies that either
$$
\mathrm{mult}_{Q}\Big(\tilde{\Omega}\cdot
\tilde{\Delta}_1\Big)>2\Big(1-(a_1+a_2+m-1)\Big)=4-2a_1-2a_2-2m
$$
or $\mathrm{mult}_{Q}(\tilde{\Omega}\cdot E)>2(1-a_1)$. In the
latter case, I have
$$
\mathrm{mult}_{P}\Big(\Omega\cdot\Delta_{2}\Big)\geqslant
m>2(1-a_1),
$$
since $m=\mathrm{mult}_{Q}(\tilde{\Omega}\cdot E)>2(1-a_1)$, which
is exactly what I want. Thus, to complete the proof, I may assume
that $\mathrm{mult}_{Q}(\tilde{\Omega}\cdot
\tilde{\Delta}_1)>4-2a_1-2a_2-2m$.

If $\mathrm{mult}_{P}(\Omega\cdot\Delta_{2})>2(1-a_1)$, then I am
done. Thus, to complete the proof, I may assume that
$\mathrm{mult}_{P}(\Omega\cdot\Delta_{2})\leqslant 2(1-a_1)$. This
gives me $m\leqslant 2(1-a_1)$, since
$\mathrm{mult}_{P}(\Omega\cdot\Delta_{2})\geqslant m$. Then
$$
\mathrm{mult}_{P}\Big(\Omega\cdot\Delta_{1}\Big)\geqslant m+\mathrm{mult}_{Q}\Big(\tilde{\Omega}\cdot \tilde{\Delta}_1\Big)>m+4-2a_1-2a_2-2m=4-2a_1-2a_2-m>2(1-a_2),%
$$
because $m\leqslant 2(1-a_1)$. This completes the proof.
\end{proof}

Let me show how to prove Lemma~\ref{lemma:cubic-surface} using
Theorem~\ref{theorem:main}. This is very easy.

\begin{proof}[Local proof of Lemma~\ref{lemma:cubic-surface}]
Let me use the assumptions and notation of
Lemma~\ref{lemma:cubic-surface}. I write $T_P=L+M+N$, where $L$,
$M$, and $N$ are lines on the cubic surface $S$. Without loss of
generality, I may assume that the line $N$ does not pass through
the point $P$. Let $D$ be any effective $\mathbb{Q}$-divisor on
the surface $S$ such that $D\sim_{\mathbb{Q}} -K_{S}$. I must show
that the log pair $(S,D)$ is log canonical at $P$. Suppose that
the log pair $(S,D)$ is not log canonical at $P$. Let me seek for
a contradiction.

The log pair $(S,D)$ is log canonical in a punctured neighborhood
of the point $P$ (use \cite[Lemma~5.3.6]{CoKoSm} or the proof of
\cite[Lemma~3.4]{Ch07a}). Put $D=aL+bM+cN+\Omega$, where $a$, $b$,
and $c$ are non-negative rational numbers and $\Omega$ is an
effective $\mathbb{Q}$-divisor on $S$ whose support contains none
of the lines $L$, $M$ and $N$. Put $m=\mathrm{mult}_P(\Omega)$.

Since $(S, L+M+N)$ is log canonical, $D\ne L+M+N$. Then there
exists a non-negative rational number $\mu$ such that the divisor
$(1+\mu)D-\mu T_P$ is effective and its support does not contain
at least one components of $T_P=L+M+N$. Thus, replacing $D$ by
$(1+\mu)D-\mu T_P$, I can assume that at least one number among
$a$, $b$, and $c$ is zero. On the other hand, I know that
$$
\mathrm{mult}_{P}(D)=m+a+b>1,
$$
because the log pair $(S,D)$ is not log canonical at $P$. Thus, if
$a=0$, then
$$
1=L\cdot
D\geqslant\mathrm{mult}_{P}\big(L\big)\mathrm{mult}_{P}\big(D\big)=\mathrm{mult}_{P}\big(D\big)=m+b>1,
$$
which is absurd. This shows that $a>0$. Similarly, $b>0$.
Therefore, $c=0$. Then
$$
1=N\cdot D=N\cdot(aL+bM+\Omega)=a+b+N\cdot\Omega\geqslant a+b,
$$
which implies that $a+b\leqslant 1$. On the other hand, I know
that
$$
\left\{\aligned
&1=L\cdot\Big(aL+bM+\Omega\Big)=-a+b+L\cdot \Omega\geqslant -a+b+m,\\
&1=M\cdot\Big(aL+bM+\Omega\Big)=a-b+M\cdot \Omega\geqslant a-b+m,\\
\endaligned
\right.
$$
which implies that $m\leqslant 1$. Thus, I can apply
Theorem~\ref{theorem:main} to  $(S,aL+bM+\Omega)$. This gives
either
$$
1+a-b=\mathrm{mult}_{P}(\Omega\cdot L)>2(1-b)
$$
or $1-a+b=\mathrm{mult}_{P}(\Omega\cdot M)>2(1-a)$. Then either
$1+a-b>2-2b$ or $1-a+b>2-2a$. In both cases, $a+b>1$, which is not
the case (I proved this earlier).
\end{proof}

I was very surprised to find out that Theorem~\ref{theorem:main}
has many other applications as well. Let me show how to use
Theorem~\ref{theorem:main} to give a short proof of
Lemma~\ref{lemma:Kosta}.

\begin{proof}[{Proof of Lemma~\ref{lemma:Kosta}}]
Let me use the assumptions and notation of
Lemma~\ref{lemma:Kosta}. Every cuspidal curve in $|-K_{S}|$ is a
 Siberian tigers, since all curves in $|-K_{S}|$ are
irreducible. Let $D$ be a  Siberian tigers. I must prove that $D$
is a cuspidal curve in $|-K_{S}|$.

The pair $(S,D)$ is not log canonical at some point $P\in S$. Let
$C$ be a curve in $|-K_{S}|$ that contains $P$. If $P$ is the base
locus of the pencil $|-K_{S}|$, then $(S,C)$ is log canonical at
$P$, because every curve in the pencil $|-K_{S}|$ is smooth at its
unique base point. Moreover, if $P=O$, then $(S,C)$ is also log
canonical at $P$ by \cite[Theorem~3.3]{Park}. In the latter case,
the curve $C$ has an ordinary double point at $P$ by
\cite[Theorem~3.3]{Park}, which also follows from Kodaira's table
of singular fibers of elliptic fibration. Furthermore, if $C$ is
singular at $P$ and $(S,C)$ is not log canonical at $P$, then $C$
has an ordinary cusp at $P$.

If $D=C$ and $C$ is a cuspidal curve, then I am done. Thus, I may
assume that this is not the case. Let me seek for a contradiction.

I claim that $C\not\subseteq\mathrm{Supp}(D)$. Indeed, if $C$ is
cuspidal curve, then $C\not\subseteq\mathrm{Supp}(D)$, since $D$
is a  Siberian tiger. If $(S,C)$ is log canonical, put
$D=aC+\Omega$, where $a$ is a non-negative rational number, and
$\Omega$ is an effective $\mathbb{Q}$-divisor on $S$ whose support
does not contain the curve $C$. Then $a<1$, since
$D\sim_{\mathbb{Q}} C$ and $D\ne C$. Then
$$
\frac{1}{1-a}D-\frac{a}{1-a}C=\frac{1}{1-a}(aC+\Omega)-\frac{a}{1-a}C=\frac{1}{1-a}\Omega\sim_{\mathbb{Q}} -K_{S}%
$$
and the log pair $(S,\frac{1}{1-a}\Omega)$ is not log canonical at
$P$, because $(S,C)$ is log canonical at $P$, and $(S,D)$ is not
log canonical at $P$. Since $D$ is a  Siberian tiger, I see that
$a=0$, i.e. $C\not\subset\mathrm{Supp}(D)$.

If $P\ne O$, then
$$
1=C\cdot D\geqslant \mathrm{mult}_{P}(D),
$$
which is impossible by \cite[Excercise~6.18]{CoKoSm}, since  the
log pair $(S,D)$ is not log canonical at the point $P$. Thus, I
see that $P=O$.

Let $f\colon\tilde{S}\to S$ be a~minimal resolution of
singularities of the surface $S$. Then there are three
$f$-exceptional curves, say $E_{1}$, $E_{2}$, and $E_{3}$, such
that $E_{1}^{2}=E_{2}^{2}=E_{3}^{2}=-2$. I may assume that
$E_{1}\cdot E_{3}=0$ and $E_{1}\cdot E_{2}=E_{2}\cdot E_{3}=1$.
Let $\tilde{C}$ be the~proper transform of the~curve $C$ on
the~surface $\tilde{S}$. Then
$\tilde{C}\sim_{\mathbb{Q}}f^*(C)-E_1-E_2-E_3$.

Let $\tilde{D}$ be the~proper transform of the
$\mathbb{Q}$-divisor $D$ on the~surface $\tilde{S}$. Then
$$
\tilde{D}\sim_{\mathbb{Q}}f^*\big(D\big)-a_1E_1-a_2E_2-a_3E_3
$$
for some non-negative rational numbers $a_{1}$, $a_{2}$ and
$a_{3}$. Then
$$
\left\{\aligned
&1-a_1-a_3=\tilde{D}\cdot\tilde{C}\geqslant 0,\\
&2a_1 - a_2=\tilde{D}\cdot E_{1}\geqslant 0,\\
&2a_2 - a_1 - a_3=\tilde{D}\cdot E_{2}\geqslant 0,\\
&2a_3 - a_2=\tilde{D}\cdot E_{3}\geqslant 0,\\
\endaligned
\right.
$$
which gives $1\geqslant a_{1}+a_{3}$, $2a_1\geqslant a_2$, $3a_2
\geqslant 2a_3$, $2a_3\geqslant a_2$, $3a_2\geqslant 2a_1$,
$a_1\leqslant \frac{3}{4}$, $a_2\leqslant 1$, $a_3\leqslant
\frac{3}{4}$. On the other hand, I have
$$
K_{\tilde{S}}+\tilde{D}+\sum_{i=1}^{3}a_iE_i\sim_{\mathbb{Q}} f^{*}(K_{S}+D)\sim_{\mathbb{Q}}0,%
$$
which implies that $(\tilde{S},\tilde{D}+a_1E_1+a_2E_2+a_3E_3)$ is
not log canonical at some point $Q\in E_1\cup E_2\cup E_3$.

Suppose that $Q\in E_1$ and $Q\not \in E_2$. Then $(\tilde{S},
\tilde{D}+a_1E_1)$ is not log canonical at $Q$. Then
$2a_{1}-a_{2}=\tilde{D}\cdot E_{1}>1$ by
Theorem~\ref{theorem:adjunction}. Therefore, I have
$$
1\geqslant \frac{4}{3} a_1 \geqslant 2a_1 - \frac{2}{3} a_1\geqslant 2a_1 - a_2>1,%
$$
which is absurd. Thus, if $Q\in E_1$, then $Q=E_1\cap E_2$.
Similarly, I see that if $Q\in E_3$, then $Q=E_3\cap E_2$.

Suppose that $Q\in E_2$ and $Q\not\in E_1\cup E_3$. Then
$(\tilde{S}, \tilde{D}+a_2E_2)$ is not log canonical at $Q$. Then
$2a_{2}-a_{1}-a_{3}=\tilde{D}\cdot E_{2}>1$ by
Theorem~\ref{theorem:adjunction}. Therefore, I have
$$
1\geqslant a_{2}=2a_2-\frac{a_2}{2}-\frac{a_2}{2}\geqslant 2a_2-a_1-a_3>1,%
$$
which is absurd. Thus, I proved that either $Q=E_1\cap E_2$ or
$Q=E_3\cap E_2$. Without loss of generality, I may assume that
$Q=E_1\cap E_2$.

The log pair $(\tilde{S},\tilde{D}+a_1E_1+a_2E_2)$ is not log
canonical at $Q$. Put $m=\mathrm{mult}_{Q}(\tilde{D})$. Then
$$
\left\{\aligned
&2a_1 - a_2=\tilde{D}\cdot E_{1}\geqslant m,\\
&2a_2 - a_1 - a_3=\tilde{D}\cdot E_{2}\geqslant m,\\
&2a_3 - a_2=\tilde{D}\cdot E_{3}\geqslant 0,\\
\endaligned
\right.
$$
which implies that $a_1+a_3\geqslant 2m$. Since I already proved
that $a_{1}+a_{3}\leqslant 1$, $m\leqslant\frac{1}{2}$. Applying
Theorem~\ref{theorem:main} to the log pair
$(\tilde{S},\tilde{D}+a_1E_1+a_2E_2)$ at the point $Q$, I see that
$\tilde{D}\cdot E_1>2(1-a_2)$ or $\tilde{D}\cdot E_2>2(1-a_1)$. In
the former case, one has
$$
2a_1-a_2=\tilde{D}\cdot E_{1}>2(1-a_2),
$$
which implies that $2\geqslant 2a_1+2a_3\geqslant 2a_1+a_2>2$,
since $1\geqslant a_{1}+a_{3}$ and $2a_3\geqslant a_2$. Thus, I
proved that
$$
2a_2 - a_1 - a_3=\tilde{D}\cdot E_{2}>2(1-a_1),
$$
which implies that $2a_2+a_1>2+a_3$. Then
$2a_2+1-a_3>2a_2+a_1>2+a_3$, since $a_{1}+a_{3}\leqslant 1$. The
last inequality implies that $2a_2>1+2a_3$. Since I already proved
that $2a_3\geqslant a_2$ , I conclude that $2a_2>1+a_2$, which is
impossible, since $a_1\leqslant 1$. The obtained contradiction
completes the proof.
\end{proof}

Similarly, I can use Theorem~\ref{theorem:main} instead of
Theorem~\ref{theorem:Kosta} in the \emph{local} proof of
Lemma~\ref{lemma:cubic-surface-2} (I leave the details to the
reader). Theorem~\ref{theorem:main} has a nice and clean inductive
proof like Theorem~\ref{theorem:adjunction} has. So, what if
Theorem~\ref{theorem:main} is the desired generalization of
Theorem~\ref{theorem:Kosta}? This may seems unlikely keeping in
mind how both theorems look like. However,
Theorem~\ref{theorem:main} does generalize
Theorem~\ref{theorem:Kosta-original}, which is the ancestor and a
special case of Theorem~\ref{theorem:Kosta}. The latter follows
from

\begin{remark}
\label{remark:main-impies-Kosta-local} Let $S$ be a surface, let
$\Delta_1$ and $\Delta_2$ be two irreducible curves on $S$ that
are both smooth at $P$ and intersect transversally at $P$. Take an
effective $\mathbb{Q}$-divisor $a_1\Delta_1+a_2\Delta_2+\Omega$,
where $a_1$ and $a_2$ are non-negative rational numbers, and
$\Omega$ is an effective $\mathbb{Q}$-divisor on the~surface $S$
whose support does not contain the curves $\Delta_1$ and
$\Delta_2$. Put $m=\mathrm{mult}_{P}(\Omega)$. Let $n$ be a
positive integer such that $n\geqslant 3$.
Theorem~\ref{theorem:Kosta-original} asserts that
$\mathrm{mult}_{P}(\Omega\cdot\Delta_{1})>2a_{1}-a_{2}$ or
$$
\mathrm{mult}_{P}\Big(\Omega\cdot\Delta_{2}\Big)>\frac{n}{n-1}a_{2}-a_{1}
$$
provided that $\frac{2n-2}{n+1}a_{1}+\frac{2}{n+1}a_{2}\leqslant
1$ and the~log pair $(S,a_{1}\Delta_{1}+a_{2}\Delta_{2}+\Omega)$
is not log canonical at $P$. On the other hand,
$\mathrm{mult}_{P}(\Omega\cdot\Delta_{1})\geqslant m$ and
$\mathrm{mult}_{P}(\Omega\cdot\Delta_{2})\geqslant m$. Thus,
Theorem~\ref{theorem:Kosta-original} asserts something non-obvious
only if
\begin{equation}
\label{equation:Kosta-local}
\left\{%
\aligned
&2a_{1}-a_{2}\geqslant m,\\%
&\frac{n}{n-1}a_{2}-a_{1}\geqslant m,\\%
&\frac{2n-2}{n+1}a_{1}+\frac{2}{n+1}a_{2}\leqslant 1.\\%
\endaligned\right.%
\end{equation}
Note that (\ref{equation:Kosta-local}) implies that $a_1\leqslant
\frac{1}{2}$, $a_2\leqslant 1$, and $m\leqslant 1$. Thus, if
(\ref{equation:Kosta-local}) holds, then I can apply
Theorem~\ref{theorem:main} to the log pair $(S,
a_1\Delta_1+a_2\Delta_2+\Omega)$ to get
$\mathrm{mult}_{P}(\Omega\cdot\Delta_{1})>2(1-a_{2})$ or
$\mathrm{mult}_{P}(\Omega\cdot\Delta_{2})>2(1-a_{1})$. On the
other hand, if (\ref{equation:Kosta-local}) holds, then
$2(1-a_{2})\geqslant 2a_{1}-a_{2}$ and
$$
2(1-a_1)\geqslant \frac{2n-2}{n+1}a_{1}+\frac{2}{n+1}a_{2}.
$$
\end{remark}

Nevertheless, Theorem~\ref{theorem:main} is not a generalization
of Theorem~\ref{theorem:Kosta}, i.e. I can not use
Theorem~\ref{theorem:main} instead of Theorem~\ref{theorem:Kosta}
in general. I checked this in many cases considered in
\cite{Ch13}. To convince the reader, let me give

\begin{example}
\label{example:new-versus-old} Put $S=\mathbb{P}^2$. Take some
integers $n\geqslant 2$ and $k\geqslant 2$. Put $r=km(m-1)$. Let
$C$ be a curve in $S$ that is given by $z^{r-1}y=x^r$, where
$[x:y:z]$ are projective coordinates on $S$. Put $\Omega=\lambda
C$ for some positive rational number $\lambda$. Let $\Delta_1$ be
a line in $S$ that is given by $x=0$, and let $\Delta_2$ be a line
in $S$ that is given by $y=0$. Put $a_1=\frac{1}{m}$ and
$a_2=1-\frac{1}{m}$. Let $P$ be the intersection point
$\Delta_1\cap\Delta_2$. Then $(S, a_1\Delta_1+a_2\Delta_2+\Omega)$
is log canonical $P$ if and only if
$\lambda\leqslant\frac{1}{m}+\frac{1}{km^2}$. Take any
$\lambda>\frac{1}{m}+\frac{1}{km^2}$ such that
$\lambda<\frac{k}{km-1}$. Then
$\mathrm{mult}_{P}(\Omega)=\lambda<\frac{2}{m}\leqslant 1$ and
$$
\mathrm{mult}_{P}\Big(\Omega\cdot\Delta_1\Big)=\lambda<\frac{k}{km-1}<\frac{2}{m}=2(1-a_2),
$$
which implies that
$$
k(m-1)+\frac{m-1}{m}>km(m-1)\lambda=\mathrm{mult}_{P}\Big(\Omega\cdot\Delta_2\Big)>2(1-a_1)=\frac{2m-2}{m}
$$
by Theorem~\ref{theorem:main}. Taking $\lambda$ close enough to
$\frac{1}{m}+\frac{1}{km^2}$, I can get
$\mathrm{mult}_{P}(\Omega\cdot\Delta_2)$ as close to
$k(m-1)+\frac{m-1}{m}$ as I want. Thus, the inequality
$\mathrm{mult}_{P}(\Omega\cdot\Delta_2)>\frac{2m-2}{m}$ provided
by Theorem~\ref{theorem:main} is not very good when $k\gg 0$. Now
let me apply Theorem~\ref{theorem:Kosta} to the log pair $(S,
a_1\Delta_1+a_2\Delta_2+\Omega)$ to get much better estimate for
$\mathrm{mult}_{P}(\Omega\cdot\Delta_2)$. Put $\alpha=1$,
$\beta=1$, $M=1$, $B=km$, $A=\frac{1}{km-1}$, and $N=0$. Then
$$
\left\{%
\aligned
&1=\alpha a_{1}+\beta a_{2}\leqslant 1,\\
&1=A(B-1)\geqslant 1,\\
&1=M\leqslant 1,\\
&0=N\leqslant 1,\\
&\frac{1}{km-1}=\alpha(A+M-1)\geqslant A^{2}(B+N-1)\beta=\frac{1}{km-1},\\
&\frac{1}{km-1}=\alpha(1-M)+A\beta\geqslant A=\frac{1}{km-1},\\
&2=2M+AN\leqslant 2.\\
\endaligned\right.\\
$$
By Theorem~\ref{theorem:Kosta},
$\mathrm{mult}_{P}(\Omega\cdot\Delta_{1})>M+Aa_{1}-a_{2}$ or
$\mathrm{mult}_{P}(\Omega\cdot\Delta_{2})>N+Ba_{2}-a_{1}$. Since
$\mathrm{mult}_{P}\big(\Omega\cdot\Delta_1\big)=\lambda<\frac{k}{km-1}=M+Aa_{1}-a_{2}$,
it follows from Theorem~\ref{theorem:Kosta} that
$$
\mathrm{mult}_{P}(\Omega\cdot\Delta_{2})>N+Ba_{2}-a_{1}=k(m-1)-\frac{1}{m}.
$$
For $k\gg 0$, the latter inequality is much stronger than
$\mathrm{mult}_{P}(\Omega\cdot\Delta_{2})>\frac{2m-2}{m}$ given by
Theorem~\ref{theorem:main}. Moreover, I can always choose
$\lambda$ close enough to $\frac{1}{m}+\frac{1}{km^2}$ so that the
multiplicity
$\mathrm{mult}_{P}(\Omega\cdot\Delta_2)=km(m-1)\lambda$ is as
close to $k(m-1)+\frac{m-1}{m}$ as I want. This shows that the
inequality
$\mathrm{mult}_{P}(\Omega\cdot\Delta_{2})>k(m-1)-\frac{1}{m}$
provided by Theorem~\ref{theorem:Kosta} is almost sharp.
\end{example}

I have a strong feeling that Theorems~\ref{theorem:Kosta} and
\ref{theorem:main} are special cases of some more general result
that is not yet found. Perhaps, it can be found by analyzing the
proofs of Theorems~\ref{theorem:Kosta} and \ref{theorem:main}.


\begin{thebibliography}{99}

\bibitem{Ch07a}
I.~Cheltsov, \emph{Log canonical thresholds of del Pezzo
surfaces}, Geom. Funct. Anal. \textbf{11} (2008), 1118--1144.

\bibitem{Ch13}
I.~Cheltsov, \emph{Two local inequalties}, to appear in Izv. Math.

\bibitem{ChK10}
I.~Cheltsov, D.~Kosta \emph{Computing $\alpha$-invariants of
singular del Pezzo surfaces}, to apeear in J. Geom.  Anal. DOI
10.1007/s12220-012-9357-6.

\bibitem{CheltsovParkWon2}
I.~Cheltsov, J.~Park, J.~Won, \emph{Affine cones over smooth cubic
surfaces}, arXiv:1303.2648 (2013).

\bibitem{ChSh08c}
I.~Cheltsov, C.~Shramov, \emph{Log canonical thresholds of smooth
Fano threefolds} (with an appendix by Jean-Pierre Demailly),
Russian Math. Surveys \textbf{63} (2008), 73--180.

\bibitem{DeKo01}
J.-P.~Demailly, J.~Koll\'ar, \emph{Semi-continuity of complex
singularity exponents and K\"ahler-Einstein metrics on Fano
orbifolds},  Ann. Sci. \'Ecole Norm. Sup. \textbf{34} (2001),
525--556.

\bibitem{CoKoSm}
A.~Corti, J.~Koll\'ar, K.~Smith, \emph{Rational and nearly
rational varieties}, Cambridge Studies in Advanced Mathematics
\textbf{92} (2004), Cambridge University Press.

\bibitem{FZ2003}
H.\,Flenner, M.\,Zaidenberg, \emph{Rational curves and rational
singularities}, Math. Z. \textbf{244} (2003), 549--575.

\bibitem{KeelMcKernan}
S.~Keel, J.~McKernan, \emph{Rational curves on quasi-projective
surfaces}, Mem. Amer. Math. Soc. \textbf{140} (1999).

\bibitem{KPZ12a}
T.\,Kishimoto, Yu.\,Prokhorov, M.\,Zaidenberg,
\emph{$\mathbb{G}_a$-actions on affine cones}, to appear in
Transformation Groups.

\bibitem{Ko09}
D.~Kosta, \emph{Del Pezzo surfaces with Du Val singularities},
Ph.D. Thesis (2009), University of Edinburgh.

\bibitem{Park}
J.~Park, \emph{A note on del Pezzo fibrations of degree $1$},
Comm. Algebra \textbf{31} (2003), 5755--5768.

\bibitem{Tian}
G.~Tian, \emph{On K\"ahler--Einstein metrics on certain K\"ahler
manifolds with $c_{1}(M)>0$},
Invent. Math. \textbf{89} (1987), 225--246.%

\end{thebibliography}
\end{document}